\documentclass[a4paper,11pt]{article}
\usepackage[all]{xy}
\usepackage{amsmath}
\usepackage{amsthm}
\usepackage{amssymb}
\usepackage{amscd}
\usepackage{graphicx}
\usepackage{epsfig}
\theoremstyle{plain}
\newtheorem{thm}{\bf Theorem}[section]
\newtheorem{proposition}[thm]{\bf Proposition}
\newtheorem{lemma}[thm]{\bf Lemma}
\newtheorem{theorem}[thm]{\bf Theorem}
\newtheorem{cor}[thm]{\bf Corollary}
\newtheorem{definition}[thm]{\bf Definition}

\newtheorem{rem}{\bf Remark}

\def\be{\begin{equation}}
\def\ee{\end{equation}}

\newcommand{\cala}{{\cal A}}
\newcommand{\calb}{{\cal B}}

\newcommand{\calh}{{\cal H}}
\newcommand{\cali}{{\cal I}}

\newcommand{\cals}{{\cal S}}

\newcommand{\Der}{{Der_K(\cala)}}
\newcommand{\Derl}{{Der_\cala(\log \cali)}}
\newcommand{\Oml}{{\Omega_\cala(\log\cali)}}
\newcommand{\Om}{{\Omega_\cala}}

\title{Logarithmic Poisson cohomology:\\ 
example of calculation and application to prequantization}
\author{J. DONGHO\\
Universit\'e d'Angers, D\'epartement de Math\'ematiques, \\
UFR Sciences, LAREMA, UMR 6093 du CNRS,\\
2 bd.~Lavoisier, 49045 Angers Cedex 01, France; \\
 University of Yaound\'e I, Faculty of Sciences, \\Department of Mathematics,\\ 
Po Box 812 Yaound\'e (Cameroon), GTGAC $^\star$ \\
E-mail:{\tt joseph.dongho@etud.univ-angers.fr}}
\begin{document}

\maketitle
\begin{abstract}
In this paper, we introduce the notions of logarithmic Poisson
structure and logarithmic principal Poisson structure; we prove
that the latter induces a representation by logarithmic derivation
of the module of logarithmic Kahler differentials; therefore, it
induces a differential complex from which we derive the notion of
logarithmic Poisson cohomology. We prove that Poisson cohomology and
logarithmic Poisson cohomology are equal when the Poisson structure
is logsymplectic. We also give and example of non logsymplectic but 
logarithmic Poisson structure for which these cohomologies are equal. 
We give and example for which these cohomologies are different. 
We discuss and modify the K. Saito definition of logarithmic forms.
The notes end with an application to a prequantization
of the logarithmic Poisson algebra: $(\mathbb{C}[x,y],
\{x,y\}=x).$
\end{abstract}
\par

\vfill
\parbox{.75\textwidth}{\hrulefill}\par
\begin{minipage}[c]{\textwidth} \small 
 \noindent  $^\star$ "Groupe de Topologie et G\'eométrie d'Afrique Centrale"\\
%Date: \today  \\
 \par\noindent

\end{minipage}
%\newpage
\section*{Introduction}
 
 The classical Poisson brackets 
\begin{equation}\label{E00}
 \{f,g\}=\underset{i=1}{\overset{n}\sum}(\dfrac{\partial f}{\partial p_i}\dfrac{\partial g}{\partial q^i}-\dfrac{\partial f}{\partial q^i}\dfrac{\partial g}{\partial p_i})
\end{equation}
 defined on the algebra of smooth functions on $\mathbb{R}^{2n},$ play a fundamental role in the analytical mechanics. They were discovered by D. Poisson in 1809. It was only a century later when A. Lichnerowicz (in \cite{AL})  and A. Weinstein (in \cite{AW}) extend it in a large theory known now as the Poisson Geometry. It has been remarked by A. Weinstein (\cite{AW}) that in fact, the theory can be traced back to S. Lie (in \cite{SL}). The Poisson bracket (\ref{E00}) is derived from a symplectic structure on $\mathbb{R}^{2n}$ and it appears as one of the main ingredients of symplectic geometry.\\

The basic properties of the bracket (\ref{E00}) are that it yields the structure of a Lie algebra on the space of functions and it has a natural compatibility with the usual associative product of functions.\\
These facts are of algebraic nature, and it is natural to define an abstract notion of a \textit{Poisson algebra}. \\
Following A. Vinogradov and I. Krasil'shchilk in \cite{VK}, J. Braconnier (in \cite{JB}) has developed the algebraic version of Poisson geometry.\\
One of the most important notion related to the Poisson geometry is the Poisson cohomology which was introduced by A. Lichnerowicz (in \cite{AL}) and in algebraic setting by I. Krasil'shchilk (in \cite{IK}). 
 Unlike the De Rham cohomology, the Poisson cohomology are almost irrelevant to a topology of the manifold.  Moreover, they have bad functorial properties and they are very large, and their actual computation is both more complicated and less significant than it is in  the case of the De Rham cohomology. However, they are very interesting because they allow us to describe various important results concerning the Poisson structures. In particular, they provide an appropriate setting for the \textit{geometric quantization} of the manifold. The algebraic aspect of this theory were developed by J. Huebschmann (in \cite{JH}) and for the geometrical setting see I. Vaisman (in \cite{IV})\\

This paper deals with Poisson algebras, but Poisson algebras of another kind. More specifically, we study the  \textit{logarithmic Poisson structures}. If the Poisson structures draw their origins from symplectic structures, logarithmic Poisson structure are inspired by log symplectic structures which are in its turn based on the theory of logarithmic differential forms. The latter were introduced by P. Deligne (in \cite{PD}) who defined it only in the case of normal crossing divisor of a given complex manifold. But it was only the theory of logarithmic differential forms along a singular divisor not necessarily normal crossings was in 1980s wen appeared in the K. Saito work's (see \cite{KS}). Explicitly, if $\cali$ is an ideal in a commutative algebra $\cala$ over a commutative ring $R,$ a derivation $D$ of $\cala$ is called logarithmic along $\cali$ if $D(\cali)\subset \cali.$ We denote by $Der_\cala(\log\cali)$ the $\cala$-module of derivations of $\cala,$ logarithmic along $\cali.$  A Poisson structure $\{-,-\}$ on $\cala$ is called logarithmic along $\cali$ if for all $a\in\cala,$ we have $\{a,-\}\in Der_\cala(\log\cali).$  
 In addition, suppose that $\cali$ is generated by $\{u_1,...,u_p\}\subset \cala$ and let $\Omega_\cala$ be the $\cala$-module of K\"alher differential. The $\cala$-module $\Omega_\cala(\log\cali)$ generated by $\{\dfrac{du_1}{u_1},...,\dfrac{du_p}{u_p}\}\cup\Omega_\cala$ is called the module of K\"alher differentials logarithmic along $\cali$.\\
 
With the above definition  we point out that the K. Saito definition of logarithmic forms is incomplete if we do not add the hypothesis that the defining function of the divisor is square free. In fact, according to K. Saito (Definition 1.2 in \cite{KS}), $\dfrac{dx}{x^2}$ and $\dfrac{dy}{x}$ are logarithmic along $D=\{(x,y)\in\mathbb{C}^2, x^2=h(x,y)=0\}.$ If that is the case, the the system $(\dfrac{dx}{x^2}, \dfrac{dy}{x})$ will a basis of $\Omega$; this is a contradiction with Theorem 1.8 in \cite{KS}; since $\dfrac{dx\wedge dy}{x^3}\neq \dfrac{unit}{x^2}dx\wedge dy.$ \\ 

 In the case where $\cali$ is generated by $\{u_1,...,u_p\}\subset \cala,$ we say that a Poisson structure $\{-,-\}$ on $\cala$ logarithmic principal along $\cali$ if for all $a\in\cala, u_i\in\{u_1,...,u_p\},$\quad $\dfrac{1}{u_i}\{a,u_i\}\in\cala.$ \\

  The J. Huebschmann program of algebraic construction of the Poisson cohomology can be  summarized as follows: \\ Let $ \cala $ be a commutative algebra over a commutative ring $R.$  A Lie-Rinehart algebra on $\cala$ is an $\cala $-module which is an $R$-Lie algebra acting on $\cala $ with suitable compatibly conditions. J. Huebschmann observes that each Poisson structure $\{-,-\}$ gives rise to a structure of Lie-Rinehart algebra in the sens of G. Rinehart (in \cite{GR}) on the $\cala$-module $\Omega_\cala$ in natural fashion. But it was proved in \cite{RP} that; anny Lie-Rinehart algebra $L$ on $\cala$ gives rise to a complex $Alt_\cala(L,\cala)$ of alternating forms which generalizes the usual De Rham complex of manifold and the usual complex computing Chevalley-Eilenberg (in \cite{CE}) Lie algebra cohomology. Moreover, extending earlier work of Hochshild, Kostant and Rosenberg (in \cite{HR}), G. Rinehart has shown that, when $L$ is projective as an $\cala$-module, the homology of the complex $Alt_\cala(L,\cala)$ may be identified with $Ext^*_{U(\cala, L)}(\cala,\cala)$ over a suitably defined universal algebra $U(\cala,L)$ of differential operators. But the latter is the Lie algebra cohomology $H^*(L, \cala)$ of $L.$ So, since $\Omega_\cala$ is free $\cala$-module, it is projective. Therefore, The homology of the complex $Alt_\cala(\Omega_\cala,\cala)$ computing the cohomology of the underline Lie algebra of the Poisson algebra $(\cala,\{-, -\}).$ Then, Poisson cohomology of $(\cala,\{-, -\})$ is the homology of $Alt_\cala(\Omega_\cala,\cala).$\\

It follows from the definition of Poisson structure that the image of Hamiltonian map of logarithmic principal Poisson structure is submodule of $Der_\cala(\log\cali)$.
Inspired by this fact, we introduce the notion of logarithmic Lie-Rinehart structure. So, a Lie-Rinehart algebra $L$ on $\cala$ is saying logarithmic along an ideal $\cali$ of $\cala$ if it acts by logarithmic derivations on $\cala$. \\

In the case of logarithmic principal Poisson structure, we replace in the J. Huebschmann program's $\Omega_\cala$ by $\Omega_\cala(\log\cali)$ and we prove the following result:
\begin{enumerate}
 \item[$\bullet$] For all logarithmic principal Poisson structure,  $\Omega_\cala(\log\cali)$ is a logarithmic Lie-Rinehart algebra.
\end{enumerate}
From this result, we define logarithmic Poisson cohomology as homology of the complexe $Alt_\cala(\Omega_\cala(\log\cali),\cala).$\\
We also prove that
\begin{enumerate}
 \item [$\bullet$] Poisson cohomology and logarithmic Poisson cohomology are equal when the Poisson structure is log symplectic.
\end{enumerate}
We check this result on the example, $(\cala=\mathbb{C}[x,y], \{x,y\}=x).$
We also show that  the logarithmic principal Poisson algebra. $(\cala=\mathbb{C}[x,y], \{x,y\}=x^2)$ is not log symplectic 
but its Poisson cohomology is equal to its logarithmic Poisson cohomology.\\  They are different in general and we show that for  $(\cala=\mathbb{C}[x,y,z], \{x,y\}=0, \{x,z\}=0, \{y,z\}=xyz),$ one can prove that:
\begin{enumerate}
 \item [$\bullet$] Its $3^{rd}$ Poisson cohomology is
\begin{equation*}
\begin{array}{ccc}
 H^3_P\cong \mathbb{C}[y]\oplus z\mathbb{C}[z]\oplus x\mathbb{C}[x]\oplus xy\mathbb{C}[y]\oplus xy\mathbb{C}[x]\oplus\\
xz\mathbb{C}[x]\oplus xz\mathbb{C}[z]\oplus yz\mathbb{C}[y]\oplus yz\mathbb{C}[z]
\end{array}
\end{equation*}
\item[$\bullet$] Its $3^{rd}$ logarithmic Poisson cohomology is 
\begin{equation*}\label{E7}
 H^3_{PS}\cong  \mathbb{C}[y]\oplus z\mathbb{C}[z]\oplus x\mathbb{C}[x]
\end{equation*}
\end{enumerate}
The structure of paper is as following. It consists of 4 sections:
\begin{enumerate}
 \item [Section 1] In this section, we introduce the notions of principal Poisson structures and logarithmic Poisson cohomology. For this, we use the notions of Lie-Rinehart algebra and logarithmic-Lie-Rinehart algebra. The main results of this section
 are Theorem 1.10 and Corollary 1.13 of Proposition 1.12.
\item[Section 2] We recall the notion of log symplectic manofold and we prove that Poisson structure induced by log symplectic structure is logarithmic principal Poisson structure.
\item[Section 3] In this section, we compute Poisson cohomology and logarithmic Poisson cohomology of 3 logarithmics principal Poisson structure. Thanks to the Theorem 3.14, we show that in general, these two cohomologies are different.
\item[Section 4] We apply logarithmic Poisson cohomology to a prequantization of the logarithmic principal Poisson structure $\{x,y\}=x.$
\end{enumerate}

\section{Logarithmic Poisson cohomology.}
\subsection{Notations and conventions.}
Throughout this paper, $R$ denote a commutative ring, $\cala$ is a
commutative, unitary $R$-algebra, $Der_\cala$ is the $\cala$-module
of derivations of $\cala$ and $\Omega_\cala$ is the $\cala$-module
of Kalher differentials.
 An action of a Lie $R$-algebra $L$ on $\cala$ is a morphism of Lie algebras $\rho:L\rightarrow
 Der_\cala.$ For all
$R$-module $M,$ an action of a Lie $R$-algebra $L$ on $M$ is a
morphism of Lie algebras $r:L\rightarrow End_R(M).$
\subsection{ Poisson cohomology.}
 Let $L$ be a Lie algebra over $R.$
 A structure of Lie-Rinehart\footnote{see \cite{GR} or \cite{JH}} algebra on $L$ is an action $\rho:L\rightarrow Der_\cala$ of
 $L$ on $\cala$ satisfying the following
compatibility properties:
\begin{enumerate}
 \item $[\rho(al)](b)=a(\rho(l)(b))$
\item  $[l_1,al_2]=\rho(l_1)(a)l_2+a[l_1,l_2]$
\end{enumerate}
%\end{definition}
A Lie-Rinehart algebra is a pair $(L,\rho)$ where $\rho$ is a
structure of Lie-Rinehart algebra on $L.$ In the sequel, all
Lie-Rinehart  algebra $(L,\rho)$ is denoted simply by $L$ if no
confusion is possible. Let $Alt^p_\cala(L,\cala)$ be the $R$-module
of alternating p-forms on a Lie-Rinehart algebra $L.$
 The following map
\begin{equation*}
 \begin{array}{llll}
  d_\rho(f)(l_1,...,l_p)&=&\underset{i=1}{\overset{p}\sum}(-1)^{i+1}\rho(\alpha_i)f(l_1,...,\hat{l_i},....,l_p)\\
&+& \underset{i<j}{\sum}(-1)^{i+j}f([l_i,l_j],
l_1,...,\hat{l_i},...,\hat{l_j},...,l_p)
 \end{array}
\end{equation*}
induces a structure of a chain complex on
$Alt_\cala(L,\cala):=\underset{p\geq
0}{\bigoplus}Alt^p_\cala(L,\cala)$ and the associated cohomology
is called Lie-Rinehart cohomology of $L.$\\
It is known that for each Poisson algebra $(\cala, \{-,-\})$, the
following data:
 \begin{enumerate}
  \item Lie-Poisson bracket $[da, db]:=d\{a,b\}$ on $\Omega_\cala.$
\item Hamiltonian map $H:\Omega_\cala\rightarrow Der_\cala$, defined by $H(da)b:=\{a,b\}.$
 \end{enumerate}
induces a Lie-Rinehart structure on the $\cala$-module
$\Omega_\cala.$ The associated Lie-Rinehart cohomology is called
Poisson cohomology of $(\cala, \{-,-\})$ and the corresponding
cohomology space is denoted by $H^*_P.$

\subsection{Logarithmic Poisson cohomology.}
Let $\cali$ be a non trivial ideal of $\cala$ and $L$ a Lie algebra
over $R$ who is also an $\cala$-module. For all $\delta\in
Der_\cala,$ we say that:
\begin{enumerate}
 \item $\delta$ is logarithmic along $\cali$ if $\delta(\cali)\subset\cali.$
\item $\delta$ is  logarithmic principal along $\{u_1, ..., u_p\}\subset
\cali$ if for all $i=1, ..., p$\\ $\delta(u_i)\in u_i\cala.$
\end{enumerate}
%\end{definition}
We denoted by $Der_\cala(\log \cali)$ the $\cala$-module of
derivations of $\cala$ logarithmic along $\cali$ and
$\widehat{Der_\cala(\log \cali)}$ the module of logarithmic
principal derivations on $\cala.$ It is clea that $Der_\cala(\log
\cali)$ is a submodule of $Der_\cala.$ Among the structures of
Lie-Rinehart algebra $\rho:L\rightarrow Der_\cala$ on $L$, there are
those whose image lives in $Der_\cala(\log \cali)$.
\begin{definition}
 A Lie-Rinehart structure $\rho:L\rightarrow Der_\cala$ on $L$ is saying logarithmic along $\cali$ if
 $\rho(L)\subset Der_\cala(\log \cali).$
\end{definition}
Let $L$ be a logarithmic Lie-Rinehart algebra.
\begin{definition}
A logarithmic Lie-Rinehart cohomology of $L$ is the Lie-Rinehart
cohomology associated to the representation of $L$ by logarithmic
derivations along $\cali$.
\end{definition}

 By the definition,
$Der_\cala(\log \cali)$ is an logarithmic Lie-Rinehart algebra. Let
$(L,\rho)$ be a logarithmic Lie-Rinehart algebra
%\begin{ex}
%Let $X$ be a dimension $n$ complex manifold and $D$ et reduced
%divisor of $X.$ The $\mathcal{O}_X$-module $Der_X(\log D)$ of vector
%fields logarithmics along $D$ is a logarithmic Lie-Rinehart algebra
%along the ideal $\cali_D$ of $D.$
%\end{ex}
 we denoted by $(Alt(L, \cala),d_\rho)$ the complex induced by its action $\rho$ on $\cala.$\\
As in the case of Lie-Rinehart algebra, the notion of
logarithmic-Lie-Rinehart-Poisson and
logarithmic-Lie-Rinehart-symplectic structures are well defined.

Let $(L, \rho)$ be a logarithmic Lie-Rinehart algebra.
\begin{definition}  A logarithmic-Lie-Rinehart-Poisson structure on
$(L, \rho)$ is a skew-symmetric 2-form $\mu:L\times
L\rightarrow\cala$ such that $d_{\rho}\mu=0.$
\end{definition}
A logarithmic-Lie-Rinehart-Poisson algebra is a triple $(L, \rho,
\mu)$ where $\mu$ is a logarithmic-Lie-Rinehart-Poisson structure on
$(L, \rho).$
\begin{definition}
A logarithmic Lie-Rinehart-Poisson algebra  $(L, \rho, \mu)$  is
called  logarithmic Lie-Rinehart-symplectic if the 2-form $ \mu$ is
non degenerate. In other words, the map $$I:
L\rightarrow\mathcal{H}om(L, \cala), \quad  l\mapsto I(l)=i_l\mu$$
is an isomorphism of $\cala$-modules. Where for all $l\in L,$ the
map
\begin{equation*}
 i_l:Alt(L,\cala)\rightarrow Alt(L,\cala)
\end{equation*}
is defined by
\begin{equation*}
 (i_l(f))(l_1,...,l_{p-1})=f(l,l_1,...,l_{p-1})
\end{equation*}
\end{definition}
Let $\cals:=\{u_1, ..., u_p\}\subset \cala$ such that $u_i\cala$ are prime ideal and $u_i\notin u_j\cala$ for all $i\neq j; i, j=1,...,p.$
We denoted by $\Omega_\cala(\log \cali)$ the $\cala$-module generated by
$\{\dfrac{du_i}{u_i}; i=1,...,p\}\cup\Omega_\cala.$
\begin{definition}
 $\Omega_\cala(\log \cali)$ is called $\cala$-module of Kalher's logarithmic  differentials on $\cala.$
\end{definition}
The following Proposition give the dual of the $\cala$-module
$\Omega_\cala(\log \cali).$
\begin{proposition}The $\cala$-module of $\cala$-linear maps from
$\Omega_\cala(\log \cali)$ to $\cala$ is isomorphic to the
$\cala$-module
 $\widehat{Der_\cala(\log \cali)}$ of logarithmic principal
 derivations.
\end{proposition}
\begin{proof}
From the universal property of $(\Omega, d);$ there is an
isomorphism $\sigma$ from $Der_\cala$ \quad to \quad
$\mathcal{H}om(\Omega_\cala,\cala).$ Consider
\begin{equation*}
 \hat{\sigma}:\widehat{Der_\cala(\log \cali)}\rightarrow\mathcal{H}om(\Omega_\cala(\log\cali),\cala)
\end{equation*}
defined by
$\hat{\sigma}(\delta)(a\dfrac{du_i}{u_i}+bdc)=a\dfrac{1}{u}\sigma(\delta)(du)+b\sigma(\delta)(dc).$
We see from a straightforward computation that $\hat{\sigma}$ is an
isomorphism.
\end{proof}
 Let $(\cala,\{-,-\})$ be a Poisson algebra and $\cals$ as above.
\begin{definition}
We say that $(\{-,-\})$ is:
\begin{enumerate}
 \item a logarithmic Poisson structure along $\cali$ if for all $a\in\cala,$  $\{a,-\}\in Der_\cala(\log \cali).$
\item  a logarithmic principal Poisson structure along $\cals$ if for all $a\in\cala,$  $\{a,-\}\in \widehat{Der_\cala(\log \cali)}.$
\end{enumerate}
\end{definition}
When $\cala$ is endowed with  a logarithmic Poisson structure along
$\cali$ (respectively a logarithmic principal Poisson structure
along $\cals$), we say that $(\cala,\{-,-\})$ is a logarithmic
(respectively a logarithmic principal )Poisson algebra.
\begin{proposition}
 Let  $(\cala,\{-,-\})$ be a Poisson algebra
\begin{enumerate}
 \item If $(\{-,-\})$ is logarithmic along $\cali,$ then $H(\Omega_\cala)\subset Der_\cala(\log D).$
\item If $(\{-,-\})$ is logarithmic principal along $\cals,$ then $H(\Omega_\cala)\subset \widehat{Der_\cala(\log D})$ and $H$ extended to
$\Omega_\cala(\log \cali)$ by $$\tilde{H}:\Omega_\cala(\log \cali)\rightarrow \widehat{Der_\cala(\log D)}; \quad \dfrac{du}{u}\mapsto \dfrac{1}{u}H(du)$$ for all $u\in\cals$
\end{enumerate}
\end{proposition}
\begin{proof}
 The first item follows from the definition of a logarithmic Poisson structure.\\
To prove item 2, we shall remark that, if $\{-,-\}$ is a logarithmic
principal Poisson structure on $\cala,$ then for all $i\neq j,
\dfrac{1}{u_iu_j}\{u_i,u_j\}\in\cala.$
\end{proof}
Let  $(\cala,\{-,-\})$ be a logarithmic principal Poisson algebra.
\begin{definition}
$\tilde{H}$ is called logarithmic Hamiltonian map of
$(\cala,\{-,-\})$.
\end{definition}
We define on $\Omega_\cala(\log \cali)$ the following bracket:
\begin{equation*}
 [a\dfrac{du_i}{u_i}, bdc]_s=\dfrac{a}{u_i}\{u_i,b\}dc+ b\{a,c\}\dfrac{du_i}{u_i}+abd(\dfrac{1}{u_i}\{u_i,c\})
\end{equation*}
\begin{equation*}
 [a\dfrac{du_i}{u_i}, b\dfrac{du_j}{u_j}]_s=\dfrac{a}{u_i}\{u_i,b\}\dfrac{du_j}{u_j}+\dfrac{b}{u_j}\{a,u_j\}\dfrac{du_i}{u_i}+abd(\dfrac{1}{u_iu_j}\{u_i,u_j\})
\end{equation*}
\begin{equation*}
 [adc, bde]_s=a\{c,b\}de+ b\{a,e\}dc+abd(\{c,e\})
\end{equation*}
for all $u_i, u_j\in\cals$ and $a, b, c, e\in\cala-\cals.$
\begin{thm} For all  logarithmic principal Poisson algebra $(\cala,\{-,-\})$,
\begin{enumerate}
 \item $[-,-]_s$ is a Lie bracket.
\item $\tilde{H}$ is logarithmic Lie-Rinehart structure on $\Omega_\cala(\log \cali).$
\end{enumerate}
\end{thm}
\begin{cor}
Each logarithmic Poisson structure along $\cali$ (logarithmic
principal Poisson structure along $\cals$ ) on $\cala$ induces a
logarithmic-Lie-Rinehart-Poisson structure $\mu$ on
$\Omega_\cala(\log \cali).$
\end{cor}
Given a logarithmic principal Poisson structure $\{-,-\}$ on $\cala$
and $\mu$ the associated logarithmic-Lie-Rinehart-Poisson structure
we have:
\begin{proposition}\label{P1.10}
 $\mu$ is a logarithmic-Lie-Rinehart-symplectic structure if and only if $\tilde{H}$ is an isomorphism.
\end{proposition}
 \begin{proof}
  Suppose that $\tilde{H}$ is an isomorphism.\\
Let $x, y\in \Omega_\cala(\log\cali)$ such that $I(x)=I(y).$ Then \\
$-\hat{\sigma}(\tilde{H}(x))=-\hat{\sigma}(\tilde{H}(y)).$ Therefore, $x=y$ and we conclude that $I$ is an monomorphism.\\
Let $\psi\in\mathcal{H}om(\Omega_\cala(\log\cali)),$ we seek $x\in \Omega_\cala(\log\cali)$ such that; $I(x)=\psi.$\\
Since $\psi\in\mathcal{H}om(\Omega_\cala(\log\cali)),\quad \hat{\sigma}^{-1}(\psi)\in\widehat{Der_\cala(\log\cali)}.$ Therefore, there is
$z\in \Omega_\cala(\log\cali)$ such that $\tilde{H}(z)=\sigma^{-1}(\psi);$ \quad i.e;\quad
$I(-z)=\hat{\sigma}(\tilde{H}(z))=\psi.$ Just take $x=-z.$ \\
Conversely, we suppose that $I$ is an isomorphism and we shall prove that $\tilde{H}$ is also an isomorphism.\\
If $\tilde{H}(x)=\tilde{H}(y),$ then $-\hat{\sigma}(\tilde{H}(x))=-\hat{\sigma}(\tilde{H}(y))$ i.e; $I(x)=I(y).$
Then $x=y.$\\
For all $\delta\in\widehat{Der_\cala(\log\cali)}$, there is
$x\in\Omega_\cala(\log\cali)$ such that; $\hat{\sigma}(\delta)=I(x)=-\hat{\sigma}(\tilde{H}(x)).$
 \end{proof}
Let $f\in\Omega^p_\cala(\log\cali)$ we define $\tilde{H}(f)\in Alt^{p}(\Omega_\cala(\log\cali),\cala)$ by\\
$\tilde{H}(f)(\alpha_1,...,\alpha_p):=(-1)^pf(\tilde{H}(\alpha_1),...,\tilde{H}(\alpha_p)).$
\begin{cor}
 If $(\cala, \{-,-\})$ is a logarithmic principal Poisson algebra, then
\begin{equation*}
 d_{\tilde{H}}\circ \tilde{H}=-\tilde{H}\circ d
\end{equation*}
\end{cor}
\begin{definition}
Let $(\cala, \{-,-\})$ be a logarithmic principal Poisson algebra
along an ideal $\cali.$ We call logarithmic Poisson cohomology the
Lie-Rinehart logarithmic cohomology associated to the action
$\tilde{H}:\Omega_\cala(\log \cali)\rightarrow
Der_\cala(\log\cali).$\\ We write $H^*_{PS}$ for the associated
cohomology space.
\end{definition}
%%%%%%%%%%%%%%%%%%%
%%%%%%%%%%%%%%
 Let $\mu\in \bigwedge^2Der(\log\cali)$ be a log symplectic
structure on $\cala.$ According to the definition of
logarithmic-Lie-Rinehart-symplectic structure, the above map $I$
defines an isomorphism; which induces an isomorphism between Poisson
cohomology $H^*_P$ and logarithmic De Rham cohomology
$H^*_{DS}.$\footnote{Where DS means De Rham Saito.} In other hand,
the above proposition proves that $\tilde{H}$ is an isomorphism
between logarithmic Poisson cohomology $H^*_{PS}$\footnote{Where PS
means Poisson Saito} and logarithmic De Rham cohomology $H^*_{DS}.$\\
Therefore, we have the following diagram of chain complex.
\[\xymatrix{(\Omega^*_\cala(\log\cali), d)\ar[rr]^{\cong}\ar[drr]_{\cong}&&(Der^*_\cala(\log\cali), d_H)\ar[d]^{\cong}\\
&&(Der^*_\cala(\log\cali), d_{\tilde{H}})}\]
We conclude that:
\begin{cor}
 If $\mu\in \bigwedge^2Der(\log\cali)$ is a log symplectic structure on $\cala,$ then
\begin{equation*}
 H^*_P\cong H^*_{DS}\cong H^*_{PS}
\end{equation*}
\end{cor}
\section{Log symplectic manifold.}
It is well known that the first examples of Poisson manifolds are
symplectics manifolds. In this section, we recall the notion of log
symplectic manifold and we prove that their induce a logarithmic
Poisson manifolds. Of cause, we need to recall the notion of
logarithmic forms. In this section, $X$ denote a final dimensional complex
manifold and $h$ a holomorphic map on $X.$
\begin{definition}
$h$ is square free if each factor of $h$ is simple.
\end{definition}
Let $D$ be a divisor of $X$ defined by a square free holomorphic
function $h.$
\begin{definition}
A meromorphic p-forme $\omega$ is saying logarithmic along $D$ if
$h\omega$ and $hd\omega$ are holomorphic forms.
\end{definition}
We denote $\Omega^p_X(\log D)$ the $\mathcal{O}_X$-module of
logarithmic p-formes on $D.$\\ As in \cite{KS}, a vector field $\delta$ is logarithmic along $D$ if $\delta(h)=h\mathcal{O}_X.$ We denote $\mathfrak{X}_X(\log D)$ the module of logarithmic vector fields on $X.$ 
\begin{rem}
According to our definition of logarithmic forms, $\dfrac{dy}{x}$ is
not logarithmic along the divisor $D$ defined by the set of zeros of
$x^2$ in $\mathbb{C}^2$ because the square free defining function of
$D$ is $x$ and we have $xd(\dfrac{dy}{x^2})=x(\dfrac{dx\wedge
dy}{x^2})=\dfrac{dx\wedge dy}{x}$ who is not holomorphic 2-form. But
follow K. Saito definition of logarithmic forms (see \cite{KS}
Definition 1.2 ) and consider $x^2$ as defining function of $D,$ we
have:\\ $x^2(d(\dfrac{dy}{x^2})=x(\dfrac{dx\wedge
dy}{x^2}))=dx\wedge dy\in\Omega^2_X.$ An then $\dfrac{dy}{x}$ is
logarithmic form. Moreover, this imply that $\{\dfrac{dx}{x^2},
\dfrac{dy}{x}\}$ is free base of $\Omega_X(\log D).$ This contradict
item i) of Theorem 1.8 in \cite{KS} since $\dfrac{dx}{x^2}\wedge
\dfrac{dy}{x}=\dfrac{1}{x^3}dx\wedge dy\neq\dfrac{unit}{x^2}dx\wedge
dy.$ Therefore, we shall add hypothesis square free in K. Saito
definition in \cite{KS}.
\end{rem}
 In addition, we suppose that $dim_\mathbb{C}X=2n$ and $X$ is compact.
\begin{definition}\cite{RG}
A pair $(X, D)$ is a log symplectic manifold if there is a
logarithmic 2-form $\omega\in\Omega^2_X(\log D)$ satisfying\\
$d\omega=0,$ \quad and  \quad
$\overset{n}{\overbrace{\omega\wedge\omega\wedge...\wedge\omega}}\neq
0\in H^{2n}(X,\Omega^{*}([D])).$
\end{definition} From this definition, we deduce the following lemma.
\begin{lemma}
Let $(X,D)$ be a log symplectic manifold with log symplectic 2-form $\omega.$
The map $\omega^\flat:\mathfrak{X}_X(\log D)\rightarrow \Omega_X(\log D)$\quad 
$\delta\mapsto i_\delta\omega$ is a quasi-isomorphism between Poisson cohomology and logarithmic De Rham cohomology of $X.$
\end{lemma}
\begin{proof}
It follow from the fact that $\omega$ is non degenerated.
\end{proof}
From this lemma, it follows that for all $f,g\in\mathcal{O}_X,$ there is unique $X_f, X_g\in\mathfrak{X}_X(\log D)$ such that $\omega^\flat(X_f)=df$ and $\omega^\flat(X_g)=dg.$ Therefore, the following bracket 
$\{f,g\}:=\omega(X_f, X_g)$ is well defined.
\begin{proposition}
Let $(X,D)$ be a log symplectic manifold. The bracket
\begin{equation}\label{E0}
\{f,g\}:=\omega(X_f, X_g)
\end{equation}
 is logarithmic principal Poisson structure on $\mathcal{O}_X.$
\end{proposition}
\begin{proof}It follow from the fact that 
for all $f\in \mathcal{O}_X,$
$\{f,-\}=i_{X_f}\omega\in \mathfrak{X}_X(\log D)$
\end{proof}
We have a logarithmic generalization of Darboux'theorem:
\begin{lemma}\cite{RG}
Let $(X,D)$ be a log symplectic manifold with a logarithmic form $\omega,$ where $D$ is a reduced divisor. There exist holomorphic coordinate $(z_0, z_1,...,z_{2n-1})$ of a neighborhood of each smooth point of smooth part of $D$ such that $\omega$ is given by
$\omega=\dfrac{dz_0}{z_0}\wedge dz_1+dz_2\wedge dz_3+...+dz_{2n-2}\wedge dz_{2n-1}.$ Where $\{z_0=0\}=D.$ We refer to these coordinates as log Darboux coordinates.
\end{lemma}
In the follow Proposition, we prove that the logarithmic Poisson cohomology of logarithmic Poisson structure (\ref{E0}) is isomorphic to logarithmic De Rham cohomology of $(X,D)$.
\begin{proposition}
If $(X, D)$ is log symplectic manifold, the the  logarithmic Hamiltonian map of associated Poisson structure is an isomorphism.
\end{proposition}
\begin{proof}
Let $M_{\tilde{H}}$ (respectively $M_H$) the matrice of $\tilde{H}$ (respectively $H$). In the log Darboux coordinates, we have:

$$M_{H}=\left( 
\begin{array}{cccccccc}
0 & -z_0 & 0 & . & . & . & 0 & 0 \\ 
z_0 & 0 &  0 & 0 & . & . & . & 0 \\ 
0 & 0 & 0 & -1 & 0 & . & . &  .\\ 
. & . & 1 & 0 & 0 & 0 & . &  .\\ 
. & . & . & . & . & . & . & . \\ 
. & . & . & . & . & . & . & . \\ 
0 & 0 & 0 & . & 0 & . & 0 & -1\\ 
0 & 0 & 0 & . & . & . & 1 & 0
\end{array} \right) $$
and then
$$M_{\tilde{H}}=\left( 
\begin{array}{cccccccc}
0 & -1 & 0 & . & . & . & 0 & 0 \\ 
1 & 0 &  0 & 0 & . & . & . & 0 \\ 
0 & 0 & 0 & -1 & 0 & . & . &  .\\ 
. & . & 1 & 0 & . & 0 & . &  .\\ 
. & . & . & . & . & . & . & . \\ 
. & . & . & . & . & . & . & . \\ 
0 & 0 & 0 & . & 0 & . & 0 & -1\\ 
0 & 0 & 0 & . & . & . & 1 & 0
\end{array} \right) $$
$M_{\tilde{H}}$ is  obviously inversible matrice. This end the prove.
\end{proof}
\section{Computation of some logarithmic Poisson cohomology.}

In this section, we compute both Poisson cohomology an logarithmic
Poisson cohomology of the following  logarithmic principal Poisson
algebra.
\begin{enumerate}
\item[i-] $(\cala:=\mathbb{C}[x,y], \{x,y\}=x).$
\item[ii-] $(\cala:=\mathbb{C}[x,y], \{x,y\}=x^2).$
\item[iii-] $(\cala:=\mathbb{C}[x,y,z], \{x,y\}=0, \{x,z\}=0, \{y,z\}=xyz).$
\end{enumerate}
We prove that the first one is a logsymplectic Poisson structure; what
implies according to Proposition\ref{P1.10} that Poisson cohomology
and logarithmic Poisson cohomology are equals for this structure. We also prove that
the second Poisson structure is not logsymplectic but we still have the 
equality between two cohomologies; therefore, being logsymplectic
is not necessary condition to have equality between Poisson and
logarithmic Poisson cohomologies. At the end, we compute the $3^{rd}$ groups of  Poisson and logarithmic Poisson cohomology of the third
Poisson structure. We show that in this case, these spaces are
different.
\subsection{Example 1: $(\cala:=\mathbb{C}[x,y], \{x,y\}=x).$}

Let us defined on  $\cala=\mathbb{C}[x,y] $ the following Poisson bracket
\begin{equation}\label{E1}
 (f,g)\mapsto\{f,g\}=x(\dfrac{\partial f}{\partial x}\dfrac{\partial g}{\partial y}-\dfrac{\partial f}{\partial y}\dfrac{\partial g}{\partial x})
\end{equation}
For all  $f\in\cala,$ the derivation
%\begin{equation}\label{E2}
$D_f:=x(\dfrac{\partial f}{\partial x}\dfrac{\partial }{\partial y}-\dfrac{\partial f}{\partial y}\dfrac{\partial }{\partial x})$
%\end{equation}
satisfy the relation
%\begin{equation}\label{E3}
$ D_f(x\cala)\subset x\cala.$
%\end{equation}
 Which means that the bracket $\{-,-\}$ defined by (\ref{E1}) is logarithmic principal Poisson bracket along the ideal $x\cala.$ The associated Hamiltonian map  $H:\Om\rightarrow\Der$ is defined  on generators of $\Omega_\cala$ by:

%\begin{equation}\label{E5}
%\begin{array}{lll}
$H(dx)=D_x=x\dfrac{\partial }{\partial y}$ and
$H(dy)=D_y=-x\dfrac{\partial}{\partial x}$
%\end{array}
%\end{equation}
\\
From these relations, we deduce the definition of associated
logarithmic Hamiltonian map
$\tilde{H}$ on generators of $\Omega_\cala(\log \cali).$\\
%\begin{equation}\label{E7}
%\begin{array}{cccc}
$\tilde{H}(\dfrac{dx}{x})=\dfrac{1}{x}H(dx)$  and
 $\tilde{H}(dy) = H(dy)$
%\end{array}
%\end{equation}
\\
In this particular case, we have the following description of
$\Omega_\cala(\log \cali).$
  \begin{lemma}\label{L1.1}
  \begin{equation}\label{E9}
\Oml\cong \cala\dfrac{dx}{x}\oplus\cala dy\cong\mathbb{C}[y]\dfrac{dx}{x}\oplus\Om.
\end{equation}
  \end{lemma}
  It follows from this lemma that for all $\alpha\in\Oml,$ there is  $a, b\in\cala$ such that
  $\alpha=a\dfrac{dx}{x}+bdy.$
It follows that $\tilde{H}$ is completely defined by the relation
\begin{equation}\label{E10}
\tilde{H}(a\dfrac{dx}{x}+bdy)=-bx\dfrac{\partial }{\partial x}+a\dfrac{\partial }{\partial y}\in Der(\log x\cala)
\end{equation}
In other hand, we have:
\begin{equation}\label{E11}
\begin{array}{cccc}
 & [\alpha^0_1\dfrac{dx}{x}+\alpha^1_1 dy,\alpha^0_2\dfrac{dx}{x}+\alpha^1_2 dy ]_s:= &  \\
 & \left(\dfrac{\alpha^0_1}{x}\{x,\alpha^0_2\}+\dfrac{\alpha^0_2}{x}\{\alpha^0_1, x\}+\alpha^1_2\{\alpha^0_1,y\}+\alpha^1_1\{y,\alpha^0_2\}\right)\dfrac{dx}{x} + & \\
 &\left( \dfrac{\alpha^0_1}{x}\{x,\alpha^1_2\}+\dfrac{\alpha^0_2}{x}\{\alpha^1_1,x\}+\alpha^1_1\{y,\alpha^1_2\}+\alpha^1_2\{\alpha^1_1,y\}\right)dy &
\end{array}
\end{equation}
\begin{lemma}\label{L1.2}
$[-,-]_s$ is  a Lie bracket on $\Oml.$
\end{lemma}
\begin{proof}
It follows from the relation lemma \ref{L1.1} that it suffices to show that this bracket is a Lie one on
 $\mathbb{C}[y]\dfrac{dx}{x}\oplus\Om.$\\
Since the following
\begin{equation}\label{E12}
[dx,dy]:=dx
\end{equation}
define a Lie bracket on $\Om$, then we need to put on $\mathbb{C}[y]\dfrac{dx}{x}$ a Lie bracket such that the following
\begin{equation}\label{E13}
\xymatrix{0\ar[r]&\Om\ar[r]&\Om\oplus\mathbb{C}[y]\dfrac{dx}{x}\ar[r]&\mathbb{C}[y]\dfrac{dx}{x}\ar[r]&0}
\end{equation}
becomes a split short sequence of Lie algebras. According to \cite{DA},
\begin{equation}\label{E14}
[\gamma_1+\beta_1,\gamma_2+\beta_2]=[\gamma_1,\gamma_2]+[\beta_1,\gamma_2]-[\beta_2,\gamma_1]+[\beta_1,\beta_2]
\end{equation}
where  $\gamma_i+\beta_i\in\Om\oplus\mathbb{C}[y]\dfrac{dx}{x}$ for  $i=1; 2.$\\
 is Lie bracket on  $\Om\oplus\mathbb{C}[y]\dfrac{dx}{x};$ if $\Om$ is Lie ideal of $\Om\oplus\mathbb{C}[y]\dfrac{dx}{x}.$ Therefore, it is sufficient
to prove that the bracket  (\ref{E14}) and  (\ref{E11}) are equal. By a simple application of the Jacobi identity $\{-,-\}$ we have the result.
\end{proof}
\begin{lemma}\label{L1.3}
For all  $\alpha=\alpha^0_1\dfrac{dx}{x}+\alpha^1_1dy, \beta=\beta^0_1\dfrac{dx}{x}+\beta^1_1dy \in\Oml$ and  $a\in\cala,$ we have
\begin{equation}\label{E20}
[\alpha,a\beta]=\tilde{H}(\alpha)(a)\beta+ a[\alpha,\beta]
\end{equation}
\end{lemma}
\begin{proof} It is a simple application of Jacobi identity of $\{-,-\}$
\end{proof}
\begin{lemma}\label{L1.4}
$\tilde{H}:\Oml\longrightarrow \Derl$ is Lie algebra homomorphism.
\end{lemma}
\begin{proof} Direct calculation.
\end{proof}
we deduce the following Proposition
\begin{proposition}\label{P1.5}
$(\Oml,[-,-],\tilde{H})$ is a Lie-Rinehart algebra
\end{proposition}
In what follows, we will give explicitly description of associated
logarithmic Poisson complex. From above description, we can identify
in this particular case $Alt^2(\Omega_\cala(\log\cali),\cala)$ with
$\cala^i:=\underset{ i}{\underbrace{\cala\times...\times\cala}}$
\[\xymatrix{0\ar[r]&\cala\ar[r]^{d^0_{\tilde{H}}}&\cala\times\cala\ar[r]^{d^1_{\tilde{H}}}&\cala\ar[r]&0}\]
Where
$ d^0_{\tilde{H}}(f)=(\partial_yf,-x\partial_xf)$ and
                $d^1_{\tilde{H}}(f_1,f_2)= \partial_yf_2+x\partial_xf_1.$\\
We verify that $d^1_{\tilde{H}}(d^0_{\tilde{H}}f)=x(\partial^2_{xy}f-\partial^2_{xy}f)=0$
\begin{proposition}
 The associated Poisson 2-form of $\{x, y\}=x$ is $\mu=x\partial_x\wedge\partial_y$ which is log symplectic structure.
\end{proposition}
\begin{proof}
 The associated log symplectic 2-form is $\omega=\dfrac{dx}{x}\wedge dy.$
\end{proof}

It follow from this Proposition that Poisson cohomology, logarithmic Poisson cohomology and logarithmic De Rham cohomology are equal.
\subsubsection{Computation of $H^i_{PS}; i=0, 1, 2.$}
These spaces are given by the following Proposition
\begin{proposition}\label{L2.1}
$H^0_{PS}\cong\mathbb{C}$, $H^1_{PS}\cong\mathbb{C}$, $H^2_{PS}\cong 0_\cala.$
\end{proposition}
\begin{proof}
According to the above construction of cochains spaces of logarithmic Poisson complex, we have:
\begin{enumerate}
\item Calculation of  $H^0_{PS}.$\\
For all $f\in\cala.$

$f\in\ker d^0_{\tilde{H}}$ iff $ \dfrac{\partial f}{\partial y}= \dfrac{\partial f}{\partial x}=0$
Therefore $f\in\mathbb{C}$

\item Calculation of  $H^2_{PS}.$\\
For all  $g\in\cala, g=d^1_{\tilde{H}}(0, \int gdy + k(x)).$
  Then $ d^1_{\tilde{H}}$  is an epimorphism
 \item Calculation of  $H^1_{PS}.$\\
 We have $\cala^2\cong(\mathbb{C}[y]\times\mathbb{C}[x])\oplus(x\cala\times y\cala).$ Then
 for all  $(f_1,f_2)\in\cala\times\cala,$ there is  $g_1\in\mathbb{C}[y], g_2\in\mathbb{C}[x], h_2,  h_1\in\cala$
such that $f_1=g_1(y)+xh_1$ and $f_2=g_2(x)+yh_2.$ But for all  $(a(y),b(x))\in \mathbb{C}[y]\times\mathbb{C}[x], x\dfrac{\partial a(y)}{\partial x}+\dfrac{\partial b(x)}{\partial y}=0.$ Then $\mathbb{C}[y]\times\mathbb{C}[x]\subset\ker d^1_{\tilde{H}}. $
For similar reasons, we have:
$$ \begin{array}{llll}
\ker (d^1_{\tilde{H}}):&=& \ker(d^1_{\tilde{H}})\cap\cala^2\\
                       &=& (\mathbb{C}[y]\times\mathbb{C}[x])\oplus \ker(d^1_{\tilde{H}})\cap(x\cala\times y\cala)\\
                       &=&  (\mathbb{C}[y]\times\mathbb{C}[x])\oplus \Theta(\cala)
\end{array}$$
where $\Theta$ is defined by
$$
\begin{array}{lllllll}
 \cala &\xymatrix{\ar[r]^\Theta &\cala^2}
                & a &\mapsto  (xa, -\int x\dfrac{\partial xa}{\partial x}dy)
\end{array}
$$
It is easy to verify that $\Theta(\cala)\subset\ker(d^1_{\tilde{H}}).$\\
In other hand, we have the following decomposition of  $\cala.$
\begin{equation*}
\cala\cong\mathbb{C}[x]\oplus y\mathbb{C}[y]\oplus xy\cala
\end{equation*}
Therefore, for all $f\in\cala,$ there is  $(f_1, q, p)\in \mathbb{C}[x]\times\mathbb{C}[y]\times\cala$
such that $f=f_1+yq+xyp.$\\
Then\\ $ \dfrac{\partial f}{\partial y}  =  q+y\dfrac{\partial
q}{\partial y}+x(p+y\dfrac{\partial p}{\partial y})
  =  (1+y\dfrac{\partial }{\partial y}) q + x(1+y\dfrac{\partial }{\partial y})p\in\mathbb{C}[y]\oplus x(1+y\dfrac{\partial }{\partial
  y})(\cala)$\\
  and\\
  $
-x\dfrac{\partial f}{\partial x}  =  -x\dfrac{\partial f_1}{\partial
x} -xyp -x^2y\dfrac{\partial p}{\partial x}
  =  -x\dfrac{\partial f_1}{\partial x} -xy(1+x\dfrac{\partial }{\partial x})p\in x\mathbb{C}[x]\oplus xy(1+x\dfrac{\partial }{\partial x})\cala
$\\
we consider;\\ $ \Psi:\cala  \rightarrow   \cala^2;\quad f \mapsto
(x(1 + y\dfrac{\partial}{\partial y})f, -xy(1+x\dfrac{\partial
}{\partial
x})f) $ \\
Since $(x(1 + y\dfrac{\partial}{\partial y})f,
-xy(1+x\dfrac{\partial }{\partial x})f)=
    ( xf\dfrac{\partial y}{\partial y}+ xy\dfrac{\partial f}{\partial y}, -x\dfrac{\partial x}{\partial x}yf -x^2\dfrac{\partial yf}{\partial x})
   =  (\dfrac{\partial xyf}{\partial y}, -x\dfrac{\partial xyf}{\partial x})
   =   d^0_{\tilde{H}}(xyf)
$ and $\Psi(\cala)\subset d^0_{\tilde{H}}(\cala).$ \\ Then
\begin{equation*}
\begin{array}{lll}
(\dfrac{\partial f}{\partial y}, -x\dfrac{\partial f}{\partial x}) & \in & (\mathbb{C}[y]\times x\mathbb{C}[x])\oplus\Psi(\cala)  \\
\end{array}
\end{equation*}
Conversely,
 for all  $F:=(f_1(y), xf_2(x))+\Psi(p)\in (\mathbb{C}[y]\times x\mathbb{C}[x])\oplus\Psi(\cala),$
As a result of the foregoing,we have
\begin{equation*}
F=d^0_{\tilde{H}}(\int f_1dy-\int f_2 dx)+ d^0_{\tilde{H}}(xyp)=d^0_{\tilde{H}}(\int f_1dy-\int f_2 dx + xyp)\in d^0_{\tilde{H}}(\cala)
\end{equation*}
Then
\begin{equation*}
d^0_{\tilde{H}}(\cala)\cong (\mathbb{C}[y]\times x\mathbb{C}[x])\oplus\Psi(\cala)
\end{equation*}
On the other hand, due to the fact that
$d^0_{\tilde{H}}(\int xady)=(xa, -\int x\dfrac{\partial xa}{\partial x}dy)$ for all $a\in\cala,$ we can conclude that$\Theta(\cala)\subset d^0_{\tilde{H}}(\cala) .$ Moreover, by direct calculation, we show that $\Theta(\cala)\subset \Psi(\cala).$\\
Since $(\mathbb{C}[y]\times\mathbb{C}[x])\cong(\mathbb{C}[y]\times
x\mathbb{C})\oplus(0_\cala\times \mathbb{C} ) $
and, $x\dfrac{\partial \cala}{\partial x}\cap\mathbb{C}=0_\cala,$ we have: $d^0_{\tilde{H}}(\cala)\cap(0_\cala\times\mathbb{C})\cong 0_\cala.$ \\
Then
$$\begin{array}{ccc}
H^1_{PS}
 & \cong &   \mathbb{C}
\end{array}$$
\end{enumerate}
\end{proof}
\subsubsection{Computation of $H^i_{DS}, i=0, 1, 2.$}
By definition, the logarithmic De Rham complex associated to the ideal $x\cala$ is:
\begin{equation}\label{E9}
 \xymatrix{0\ar[r]&\cala\ar[r]^{d^0}&\Omega^1_\cala(\log x\cala)\ar[r]^{d^1}&\Omega^2_\cala(\log x\cala)\ar[r]&0}
\end{equation}
where \begin{equation*}
      d^0(a):=x\partial_x(a)\dfrac{dx}{x}+\partial_y(a)dy; \quad d^1(a\dfrac{dx}{x}+bdy):=(x\partial_x(b)-\partial_y(a))\dfrac{dx}{x}\wedge dy.
     \end{equation*}
\begin{proposition}
The following diagram is commutative
\begin{equation*}
\xymatrix{0\ar[r]&\cala\ar[d]\ar[r]^{d^0}&\Omega_\cala(\log x\cala)\ar[d]^{-\tilde{H}}\ar[r]^{d^1}&\Omega^2_\cala(\log x\cala)\ar[d]^{-\tilde{H}}\ar[r]&0\\
          0\ar[r]&\cala\ar[r]^{d^0_{\tilde{H}}}&\cala^2\ar[r]^{d^1_{\tilde{H}}}&\cala\ar[r]&0}
\end{equation*}
\end{proposition}
\begin{proof}
For all $a\in\cala$, we have
$
%\begin{array}{lll}
\tilde{H}(da)  =  \tilde{H}(x\partial_x(a)\dfrac{dx}{x}+\partial_y(a)dy)
  =  -\partial_y(a)x\partial_x+x\partial_x(a)\partial_y
  \cong  (-\partial_y(a),x\partial_x(a))$ and
$d^0_{\tilde{H}}(a)  \cong  (\partial_y(a),-x\partial_x(a))
                    =     -\tilde{H}(da)
 %\end{array}
$
Moreover, for any $\alpha=f\dfrac{dx}{x}+gdy\in\Oml,$ we have:
$
%\begin{array}{lll}
d^1(\alpha)  =  (x\partial_x(g)-\partial_y(f))\dfrac{dx}{x}\wedge dy,$\quad
$-\tilde{H}(d^1(\alpha))  \cong  x\partial_x(g)-\partial_y(f).$\\
However,
$-\tilde{H}(\alpha)  =  gx\partial_x-f\partial_y
                    \cong  (g, -f)$
, we have
$d^1_{\tilde{H}}(-\tilde{H})  =  d^1_{\tilde{H}}(gx\partial_x-f\partial_y)
                             \cong   x\partial_x(g)-\partial_y(f)$
%\end{array}
This ends the proof
\end{proof}
The following gives the logarithmic De Rham cohomology spaces.
\begin{proposition}
$H^0_{DS}\cong\mathbb{C}$, $H^1_{DS}\cong\mathbb{C}$, $H^2_{DS}\cong 0_\cala.$
\end{proposition}
\begin{proof}
 For simplicity, we adopt the following notations:
\begin{equation*}
 \begin{array}{ccccc}
  &\Omega^1_\cala(\log x\cala)&\overset{\cong}{\rightarrow}&\cala\times\cala&\\
        &a\dfrac{dx}{x}+bdy&\mapsto&(a,b)&
 \end{array}
\begin{array}{ccccc}
  &\Omega^2_\cala(\log x\cala)&\overset{\cong}{\rightarrow}&\cala&\\
        &a\dfrac{dx}{x}\wedge dy&\mapsto&a&
 \end{array}
\end{equation*}
With these notations, the complex \ref{E9} becomes:
\begin{equation}\label{E10}
 \xymatrix{0\ar[r]&\cala\ar[r]^{d^0}&\cala\times\cala\ar[r]^{d^1}&\cala\ar[r]&0}
\end{equation}
where $d^0(f)=(x\partial_xf,\partial_yf)$, \quad and \quad $d^1(f_1, f_2)=x\partial_xf_2-\partial_yf_1.$\\
For all $f\in\cala,$ $f=d^1(-\int fdy, 0).$ Then $\cala\cong d^1(\cala\times\cala)$ and therefore, $H^2_{DS}\cong 0.$ \\
It is easy to see that $H^0_{DS}\cong\mathbb{C}.$\\
Let $(f^1, f^2)\in\cala\times\cala.$ $(f^1, f^2)\in \ker (d^1)$ iff $f^1=x\int \partial_xf^2dy+k(x).$
Then $\ker (d^1)\cong\{(x\int \partial_xudy, u); u\in\cala\}\oplus x\mathbb{C}\oplus\mathbb{C}.$
The following map is an mono morphism of vector spaces.
\begin{equation*}
 \begin{array}{ccccc}
  \theta:&\cala&\rightarrow& x\cala\times\cala&\\
          &u& \mapsto &(x\int \partial_xudy, u)&
 \end{array}
\end{equation*}
and $\ker(d^1)\cong \theta(\cala)\oplus (x\mathbb{C}\times0_\cala)\cong \theta(\cala)\oplus (x\mathbb{C}\oplus\mathbb{C}).$\\
Moreover, for any $u\in\cala$ and $a\in\mathbb{C}[x],$ we have: \\
$d^0(\int udy+\int adx)=(x\int\partial_xudy +xa,u)=(x\int\partial_xudy,u)+(xa,0)=\theta(u)+(xa,0)\in \theta(\cala)\oplus (x\mathbb{C}).$
Then $\theta(\cala)\oplus (x\mathbb{C})\subset d^0(\cala).$
Since $\mathbb{C}\cap d^0(\cala)=0_\cala$, we have; $d^0(\cala)=d^0(\cala)\cap(\ker(d^1))\cong \theta(\cala)\oplus (x\mathbb{C}).$
Therefore, $\ker(d^1)\cong d^0(\cala)\oplus\mathbb{C}.$ And then $H^1_{DS}\cong \mathbb{C}.$
\end{proof}

%It is clear that this result strengthens our expectations. It also noted that;
 %$$\tilde{H}(\mathbb{C}\dfrac{dx}{x})=\mathbb{C}\dfrac{\partial}{\partial_y}\cong H^1_{PS}.$$
\subsubsection{Computation of Poisson cohomology of\\ $\{x,y\}=x$.}

By a direct calculation, we show that the Poisson complex of  $\{x,y\}=x$ is given by:
\begin{equation}
\xymatrix{0\ar[r]&\cala\ar[r]^{d^0_H}&\cala\times\cala\ar[r]^{d^1_H}&\cala\ar[r]&0}
\end{equation}
Where
$d^0_{H}(f)=(x\partial_yf,-x\partial_xf)$ and $
d^1_{H}(f_1,f_2)=x\partial_yf_2+x\partial_xf_1-f_1$

\begin{proposition}\label{L2.2}
$H^0_P\cong \mathbb{C}$,  $H^1_P\cong \mathbb{C}$, $H^2_P\cong 0_\cala.$
\end{proposition}
\begin{proof}
It is shown without difficulty that $H^0_P\cong \mathbb{C}$ and $H^2_P\cong 0_\cala.$ So we have to prove that $H^1_P\cong \mathbb{C}.$
For this, for all $(f_1,f_2)\in \cala\times\cala,$ \\
$(f_1, f_2)\in \ker(d^1_H)$ iff there is $u\in\cala$ and $a(x)\in\mathbb{C}[x]$ such that
$(f_1, f_2)=(xu, -x\int \partial_xudy)+(0,a(x)).$ \\
We set \begin{equation*}
        \beta:\cala\rightarrow x\cala\times\cala, u\mapsto(xu,-x\int \partial_xudy)
       \end{equation*}
Clearly, $\beta$ is a monomorphism, $\ker(d^1_H)\cong\beta(\cala)\oplus x\mathbb{C}[x]\oplus\mathbb{C}$,
 $\beta(\cala)\oplus x\mathbb{C}[x]\subset d^0_H(\cala).$ In addition, there is no $f\in\cala$ such that $x\partial_xf\in\mathbb{C}^*.$
Then $\ker(d^1_H)\cong d^0_H(\cala)\oplus\mathbb{C}.$
As result, we have $H^1_P\cong \mathbb{C}.$
\end{proof}
\subsection{Example 2: $(\cala:=\mathbb{C}[x,y], \{x,y\}=x^2).$}

Let us consider on $\cala=\mathbb{C}[x,y]$ the Poisson bracket defined on variable $x, y$ by
$\{x,y\}=x^2.$\\

Note that, $\Omega_\cala(\log x^2\cala)$ is isomorphic to the $\cala$-module generated by\\
 $\{\dfrac{dx}{x}\cup\Omega_\cala\}$ since
$\dfrac{dx^2}{x^2}=2\dfrac{dx}{x}.$ Therefore, it is easy to see that the bracket $\{x,y\}=x^2$ is logarithmic principal Poisson bracket along the ideal $x^2\cala.$
the associated logarithmic Hamiltonian map is defined on generators of $\Omega_\cala(\log x^2\cala)$ by;
$\tilde{H}(\dfrac{dx}{x})=x\partial_y, \tilde{H}(dy)=-x^2\partial_x.$ We therefore deduced the associated logarithmic Poisson complex:\\
$d^0_{\tilde{H}}(f)=(x\partial_yf, -x^2\partial_xf),$ $d^1_{\tilde{H}}(f_1, f_2)=x\partial_yf_2+x^2\partial_xf_1-xf_1.$ Where we have the following identification
\begin{equation*}
 \begin{array}{ccccc}
  &Der_\cala(\log x^2\cala)&\overset{\cong}{\rightarrow}&\cala\times\cala&\\
        &ax\partial_x+b\partial_y&\mapsto&(a,b)&
 \end{array}
\begin{array}{ccccc}
  &Der_\cala(\log x^2\cala)\wedge Der_\cala(\log x^2\cala)&\overset{\cong}{\rightarrow}&\cala&\\
        &ax\partial_x\wedge\partial_y&\mapsto&a&
 \end{array}
\end{equation*}
\subsubsection{Computation of $H^2_{PS}.$}

 Since
$\cala\cong\mathbb{C}[y]\oplus x\cala,$ for all $g\in\cala,$ there
is $g_1, g_2\in\cala$ such that $g=g_1+xg_2.$ Therefore, for all
$g\in\cala, g\in d^1_{\tilde{H}}(\cala)$ iff
$g=xg_2=x\partial_yf_2+x^2\partial_xf_1-xf_1.$ But
  $xg_2=x\partial_y(x\int\partial_xg_2dy)-x^2\partial_xg_2-xg_2$ and the equation $x(\partial_yv+x\partial_xu-u)=g(y)\in\mathbb{C}[y]^*$ has no solution in $\cala\times\cala.$ Then $\cala\cong d^1_{\tilde{H}}(\cala\times\cala)\oplus\mathbb{C}[y].$  It follows that
$$H^2_{PS}\cong \mathbb{C}[y].$$
\subsubsection{Computation of  $H^1_{PS}.$}

To compute $H^1_{PS}$, we first recall the following fact.
\begin{lemma}
 Let $\varphi: E\rightarrow F$ be a mono morphism of vector spaces. For all subset $A, B$ of $E, \varphi(A\oplus B)=\varphi(A)\oplus\varphi(B)$
\end{lemma}
\begin{proof}
 It is clear that $\varphi(A\oplus B)=\varphi(A)+\varphi(B).$ If $z\in\varphi(A)\cap\varphi(B),$ then $z\in\varphi(A\oplus B)=0_E.$
Therefore, $\varphi(A\oplus B)=\varphi(A)\oplus\varphi(B).$
\end{proof}
Let $(f_1,f_2)\in\cala\times\cala.$\\
$(f_1,f_2)\in\ker(d^1_{\tilde{H}})$ iff there is $k\in\mathbb{C}[x]$
such that $f_2=\int(1-x\partial_x)f_1dy+k(x).$ So,
$\ker(d^1_{\tilde{H}})\cong\{(u, \int(1-x\partial_x)udy),
u\cala\}\oplus \mathbb{C}[x].$ We put for all $u\in\cala;
\eta(u)=(u, \int(1-x\partial_x)udy).$ Then, $\eta:\cala\rightarrow
\cala\times\cala$ is a mono morphism of vector spaces and
$\ker(d^1_{\tilde{H}})\cong\eta(\cala)\oplus\mathbb{C}[x]\cong\eta(\mathbb{C}[y])\oplus\eta(x\cala)\oplus\mathbb{C}[x];$
since $\cala\cong \mathbb{C}[y]\oplus x\cala.$ On the other hand,
for all $g\in\eta(x\cala)\oplus(0_\cala, x^2\mathbb{C}[x]),$ there
is $u\in\cala$ and $v\in \mathbb{C}[x]$ such that $g=(xu,
-x^2\int\partial_xdy+x^2v(x))=d^0_{\tilde{H}}(\int udy-\int
v(x)dx).$ Moreover, for all $u(y)\in\mathbb{C}[y]$ and $a_0,
a_1\in\mathbb{C},$ the partial differential equation:
\begin{equation*}
 \left\{\begin{array}{ccccc}
         xf_y&=&u(y)\\
         -x^2f_x&=&\int u(y)dy+a_0+a_1x
        \end{array}
\right.
\end{equation*}
has no solution in $\cala.$ Then, $\ker(d^1_{\tilde{H}})\cong \eta(\mathbb{C}[y])\oplus\mathbb{C}_1[x]\oplus d^0_{\tilde{H}}(\cala).$ Therefore,
\begin{equation*}
 H^1_{PS}\cong\eta(\mathbb{C}[y])\oplus\mathbb{C}_1[x].
\end{equation*}
where $\mathbb{C}_1[x]:=\{a_0+a_1x; a_0, a_1\in\mathbb{C}\}.$ On the other hand, since $\eta$ is a mono morphism, $\eta(\mathbb{C}[y])\cong \mathbb{C}[y].$ Then,
\begin{equation*}
 H^1_{PS}\cong\mathbb{C}[y]\oplus\mathbb{C}_1[x].
\end{equation*}
This end the prove of the following Proposition.
\begin{proposition}\label{P3.2}
 The logarithmic Poisson cohomology spaces of $\{x,y\}=x^2$ are:
\begin{equation*}
 H^1_{PS}\cong\mathbb{C}[y]\oplus\mathbb{C}_1[x]; H^2_{PS}\cong \mathbb{C}[y],H^0_{PS}\cong \mathbb{C}
\end{equation*}
\end{proposition}

\subsubsection{Poisson cohomology of
$(\cala=\mathbb{C}[x,y],\{x,y\}=x^2).$}

The action of Hamiltonian map associated to this Poisson structure on generators of $\Omega_\cala$ is:\\
$H(dx)=x^2\partial_y$ and $H(dy)=-x^2\partial_x.$\\
For the sake of simplicity, we shall use the following isomorphism:

\begin{equation*}
 \begin{array}{ccccc}
  &Der_\cala&\overset{\cong}{\rightarrow}&\cala\times\cala&\\
        &a\partial_x+b\partial_y&\mapsto&(a,b)&
 \end{array}
\begin{array}{ccccc}
  &Der_\cala\wedge Der_\cala&\overset{\cong}{\rightarrow}&\cala&\\
        &a\partial_x\wedge\partial_y&\mapsto&a&
 \end{array}
\end{equation*}
With these isomorphisms, the associated Poisson complex is giving by:
$d^0_{H}(f)=(x^2\partial_yf, -x^2\partial_xf)$ and $d^1_{H}(f_1,f_2)= x^2\partial_xf_1+x^2\partial_yf_2-2xf_1.$  For all $g\in\cala,$ we have $xg=-2x(-\dfrac{1}{2}g)+x^2(\dfrac{1}{2})(-\partial_xg+\partial_y(\int \partial_xgdy)).$ Then $\cala\cong d^1_H(\cala\times\cala)\oplus\mathbb{C}[y].$ Therefore,
\begin{equation*}
 H^2_P\cong \mathbb{C}[y].
\end{equation*}
Let $(f_1,f_2)\in\cala\times\cala$; \\
$(f_1,f_2)\in\ker(d^1_H)$ iff there is $u\in\cala, a\in\mathbb{C}[x]$ such that $f_1=xu$ and $f_2=\int(1-x\partial_x)udy+a(x).$\\ So, $\ker(d^1_H)=\{(xu,\int(1-x\partial_x)udy+a(x)),\quad u\in\cala, a(x)\in\mathbb{C}[x]\}.$ We put $\varphi(u)=(xu,\int(1-x\partial_x)udy$ for all $u\in\cala.$ Then $\varphi:\cala\rightarrow x\cala\times\cala$ is a monomorphism of vector spaces and
\begin{equation*}
 \ker(d^1_H)\cong\varphi(\cala)\oplus\mathbb{C}[x]
\end{equation*}
On other hand, since $\cala\cong\mathbb{C}[y]\oplus x\cala,$ then $\varphi(\cala)\cong\varphi(\mathbb{C}[y])\oplus \varphi(x\cala).$ Also, it is easy to prove that $\varphi(x\cala)\oplus x^2\mathbb{C}[x]\subset d^0_H(\cala), $ and that $d^0_H(\cala)\cap\varphi(\mathbb{C}[y])\oplus\mathbb{C}_1[x].$ Therefore,
\begin{equation*}
 \ker(d^1_H)\cong \varphi(\mathbb{C}[y])\oplus\mathbb{C}_1[x]\oplus d^0_H(\cala)\cong \mathbb{C}[y]\oplus \mathbb{C}_1[x]\oplus d^0_H(\cala)
\end{equation*}
Therefore,
\begin{equation*}
 H^1_P\cong\mathbb{C}[y]\oplus \mathbb{C}_1[x]
\end{equation*}
This end the prove of the following Proposition
\begin{proposition}\label{P3.3}
 The  Poisson cohomology spaces of $\{x,y\}=x^2$ are:
\begin{equation*}
 H^1_{P}\cong\mathbb{C}[y]\oplus\mathbb{C}_1[x]; H^2_{P}\cong \mathbb{C}[y],H^0_{P}\cong \mathbb{C}
\end{equation*}
\end{proposition}
\begin{rem}
 It follow from Propositions \ref{P3.3} and \ref{P3.2} that Poisson cohomology and logarithmic Poisson cohomology of the Poisson bracket $\{x,y\}=x^2$ on $\mathbb{C}[x,y]$ are equals, although the latter is not logsymplectique. Consequently, it can be concluded that being logsymplectic is not a necessary condition for equality between the Poisson cohomology spaces and logarithmic Poisson cohomology spaces. In the next section, we give an example in which the two concepts are different.
\end{rem}
\subsection{Example 3 $\cala=\mathbb{C}[x, y, z]$ and $\{x,y\}=0, \{x,z\}=0, \{y,z\}=xyz).$}
It is easy to prove that this Poisson structure is logarithmic
principal along the ideal $xyz\mathcal{A}$ and the associated
logarithmic differential is defined by:
\begin{equation}\label{E2}
\begin{array}{llll}
d^0_{\tilde{H}}(f)  =  (0, xz\dfrac{\partial f}{\partial z},  -xy\dfrac{\partial f}{\partial y})\\
d^1_{\tilde{H}}(f_1,f_2,f_3)  =  (xz\dfrac{\partial f_3}{\partial z}+ xy\dfrac{\partial f_2}{\partial y}-xf_1, -xy\dfrac{\partial f_1}{\partial y},-xz\dfrac{\partial f_1}{\partial_z})\\
d^2_{\tilde{H}}(f_1,f_2,f_3)  =  xz\dfrac{\partial f_2}{\partial
z}+xy\dfrac{\partial f_3}{\partial y}
\end{array}
\end{equation}
By definition, we have the following expressions of associated Poisson differential.
\begin{equation}\label{E3}
\begin{array}{lll}
\delta^0(f) & = & xyz(0,\dfrac{\partial f}{\partial z},-\dfrac{\partial f}{\partial y} )\\
\delta^1(f_1,f_2,f_3) & = & (xyz\dfrac{\partial f_3}{\partial z}+xyz\dfrac{\partial f_2}{\partial y}-yzf_1-xzf_2-xyf_3, -xyz\dfrac{\partial f_1}{\partial y}, -xyz\dfrac{\partial f_1}{\partial_z})\\
\delta^2(f_1,f_2,f_3) & = & xyz(\dfrac{\partial f_2}{\partial z}+\dfrac{\partial f_3}{\partial  y})
 \end{array}
\end{equation}
\subsubsection{Computation of $H^3_{PS}$}
We deduce from equations (\ref{E2}) that $d^2_{\tilde{H}}(\mathcal{A}^3)\subset x\mathcal{A}.$\\
But
\begin{equation*}
\begin{array}{ccc}
\mathcal{A}&\cong& \mathbb{C}[y]\oplus z\mathbb{C}[z]\oplus x\mathcal{A}\\
           &\cong& \mathbb{C}[y]\oplus z\mathbb{C}[z]\oplus x\mathbb{C}[x]\oplus xy\mathbb{C}[y]\oplus xz\mathbb{C}[z]\oplus x^2y\mathcal{A}\oplus
x^2z\mathcal{A}\oplus xyz\mathcal{A}.
\end{array}
\end{equation*}
On other hand, for all $xg(x)\in x\mathbb{C}[x]$ the partial differential equation
 $z\dfrac{\partial u}{\partial z}+y\dfrac{\partial v}{\partial y}=g(x)$ have no solution in $\cala\times\cala\times\cala$. Moreover, for all $g\in xy\mathbb{C}[y]\oplus xz\mathbb{C}[z]\oplus x^2y\mathcal{A}\oplus
x^2z\cala\oplus xyz\cala,$ \\
there is  $$g_1(y), g_2(z), g_3(x,y,z), g_4(x,y,z), g_5(x,y,z)\in\mathcal{A}$$
such that
$g=xyg_1(y)+ xzg_2(z)+ x^2yg_3(x,y,z)+x^2z g_4(x,y,z)+xyz g_5(x,y,z)$ \\
Therefore 2-coboundary are given by:
\begin{equation}\label{E4}
 z\dfrac{\partial f_2}{\partial_z}+y\dfrac{\partial f_3}{\partial y}=yg_1(y)+ zg_2(z)+ xyg_3(x,y,z)+xz g_4(x,y,z)+yz g_5(x,y,z)
\end{equation}
Which is equivalent to
\begin{equation}\label{E5}
 z(\dfrac{\partial f_2}{\partial_z}-g_2(z)-x g_4(x,y,z))+y(\dfrac{\partial f_3}{\partial_y}-g_1(y)-xg_3(x,y,z)-z g_5(x,y,z))=0
\end{equation}
So just take:
\begin{equation}\label{E6}
 f_2=\int g_2(z)+x g_4(x,y,z)dz;\quad f_3=\int g_1(y)+xg_3(x,y,z)+z g_5(x,y,z)dy
\end{equation}
This prove that $$d^2_{\tilde{H}}(\mathcal{A}^3)\cong xy\mathbb{C}[y]\oplus xz\mathbb{C}[z]\oplus x^2y\mathcal{A}\oplus
x^2z\mathcal{A}\oplus xyz\mathcal{A}.$$
 Therefore, we deduce that
\begin{equation}\label{E7}
 H^3_{PS}\cong  \mathbb{C}[y]\oplus z\mathbb{C}[z]\oplus x\mathbb{C}[x]
\end{equation}
\subsubsection{Computation of $H^3_{P}$} It follows from equation(
\ref{E3}) that
\begin{equation}\label{E8}
 \delta^2(\mathcal{A}^3)\subset xyz\mathcal{A.}
\end{equation}
But
\begin{equation}\label{E9}
\begin{array}{ccc}
 \mathcal{A}\cong \mathbb{C}[y]\oplus z\mathbb{C}[z]\oplus x\mathbb{C}[x]\oplus xy\mathbb{C}[y]\oplus xy\mathbb{C}[x]\oplus
xz\mathbb{C}[x]\oplus \\ xz\mathbb{C}[z]\oplus yz\mathbb{C}[y]\oplus yz\mathbb{C}[z]\oplus xyz\mathcal{A}
\end{array}
\end{equation}
 and
$$\begin{array}{ccc}
\delta^2(\mathcal{A}^3)\cap\mathbb{C}[y]\oplus z\mathbb{C}[z]\oplus x\mathbb{C}[x]\oplus xy\mathbb{C}[y]\oplus
 xy\mathbb{C}[x]\oplus\\
xz\mathbb{C}[x]\oplus xz\mathbb{C}[z]\oplus yz\mathbb{C}[y]\oplus yz\mathbb{C}[z]\cong 0_\mathcal{A}
\end{array}
$$
Since the map
\begin{equation}\label{E10}
 \mathcal{A}\times\mathcal{A}\rightarrow\mathcal{A}, (u,v)\mapsto \dfrac{\partial u}{\partial z}+\dfrac{\partial v}{\partial y}
\end{equation}
is surjective, $\delta^3(\mathcal{A}^3)\cong xyz\mathcal{A}.$\\
Therefore
\begin{equation*}\label{E11}
\begin{array}{ccc}
 H^3_P\cong \mathbb{C}[y]\oplus z\mathbb{C}[z]\oplus x\mathbb{C}[x]\oplus xy\mathbb{C}[y]\oplus xy\mathbb{C}[x]\oplus\\
xz\mathbb{C}[x]\oplus xz\mathbb{C}[z]\oplus yz\mathbb{C}[y]\oplus yz\mathbb{C}[z]
\end{array}
\end{equation*}
In conclusion, we have prove the following.
\begin{thm}
\begin{enumerate}
 \item The $3^{rd}$ Poisson cohomology of $(\cala=\mathbb{C}[x,y, z], \{x,y\}=0, \{x,z\}=0, \{y,z\}=xyz)$ is
\begin{equation*}
\begin{array}{ccc}
 H^3_P\cong \mathbb{C}[y]\oplus z\mathbb{C}[z]\oplus x\mathbb{C}[x]\oplus xy\mathbb{C}[y]\oplus xy\mathbb{C}[x]\oplus\\
xz\mathbb{C}[x]\oplus xz\mathbb{C}[z]\oplus yz\mathbb{C}[y]\oplus yz\mathbb{C}[z]
\end{array}
\end{equation*}
\item The $3^{rd}$ logarithmic Poisson cohomology of $(\cala=\mathbb{C}[x,y, z], \{x,y\}=0, \{x,z\}=0, \{y,z\}=xyz)$ is
\begin{equation}\label{E7}
 H^3_{PS}\cong  \mathbb{C}[y]\oplus z\mathbb{C}[z]\oplus x\mathbb{C}[x]
\end{equation}
\end{enumerate}
\end{thm}
\begin{rem}
 We remark the $H^3_{PS}\neq H^3_{P}.$
\end{rem}

\section{Application to prequantization of $\{x,y\}=x.$}
The problem of geometric quantization is based on the Dirac principle; which consists of representation of the underlying Lie algebra of a Poisson algebra by a Hilbert space $\mathcal{H}$. In other words, one shall  build the following commutative diagram:
\begin{equation*}
\xymatrix{0\ar[r]&\mathcal{R}\ar[d]\ar[r]&(\cala,\{-,-\})\ar[d]\ar[r]&(d\cala, [-,-]_{LP})\ar[d]\ar[r]&0\\
          0\ar[r]&\cala\ar[r]&\calh\ar[r]&\calb\ar[r]&0}
\end{equation*} where the first line is an extension of Lie algebras and the second is an extension of Lie-Rinehart algebras.
But according to (\ref{E4}) the following bracket
\begin{equation*}
 [a+\alpha, b+\beta]:=\{a,b\}+ \pi(\alpha,\beta)+[\alpha,\beta]+\tilde{H}(\alpha)b-\tilde{H}(\beta)a
\end{equation*}
is a Lie structure on $\cala\oplus\Omega_\cala(\log x\cala)$ such that
the following is an extension of Lie-Rinehart algebras
\begin{equation*}
 \xymatrix{0\ar[r]&\cala\ar[r]&\cala\oplus\Omega_\cala(\log x\cala)\ar[r]&\Omega_\cala(\log x\cala)\ar[r]&0}
\end{equation*} Where
 $\pi=x\partial_x\wedge\partial_y$ is the Poisson bivector of $\{x,y\}=x;$
By construction, $\pi$ is the associated class of this extension. \\
We consider the map $r:\cala\rightarrow \cala\oplus\Omega_\cala(\log x\cala)$ defined by
$$r(a)= a+x\partial_x(a)\dfrac{dx}{x}+\partial_y(a)dy.$$ By definition, $r$ is Lie algebra homomorphism and the following diagram commutes.
\begin{equation*}
\xymatrix{0\ar[r]&\mathbb{C}\ar[d]\ar[r]&(\cala,\{-,-\})\ar[d]^r\ar[r]&(d\cala, [-,-]_{LP})\ar[d]\ar[r]&0\\
          0\ar[r]&\cala\ar[r]&\cala\oplus\Omega_\cala(\log x\cala)\ar[r]&\Omega_\cala(\log x\cala)\ar[r]&0}
\end{equation*}
We adopted the following definition.
\begin{definition}
A Poisson structure, logarithmic along an ideal $\cali$ of  $\cala$  is saying log prequantizable if there is rang 1  projectif $\cala$-module $M$ with an $\Oml$-connection with curvature $\pi.$
 \end{definition}
According to Theorem 2.15 in \cite{JH} we have
\begin{theorem}( \cite{JH} )
Let $Pic(\cala)$ be the group of projectve rank one $\cala$-modules.
For any Lie-Rinehart algebra  $L$, the correspondence who associated to any class $[M]$ of rang 1 projectif \cala-module
 the class $[\Omega_M]\in H^2(Alt_\cala(L,\cala))$ of the curvature of associated $L$-connection of $M$ is an homomorphism.
\begin{equation*}
i: Pic(\cala)\rightarrow H^2(Alt_\cala(L,\cala))
\end{equation*}
of $\mathcal{R}$-modules.
\end{theorem}
It follow from this theorem that the logarithmic Poisson structure $\{x,y\}=x$ is log prequantizable if and only if the logarithmic Poisson cohomology class of $\pi$ is element of the image of  $i.$ \\
But according to lemma \ref{L2.1}, we have  $[\pi]\in H^2_{PS}\cong 0.$ Then $\{x,y\}=x$ is log prequantizable Poisson structure.
\newpage
\section*{\textbf{Acknowledgments}}
The author is grateful to Tagne Pelap Serge Rom\'eo, Michel Granger, Michel Nguiffo Boyom, Jean-Claude Thomas and Eug\`ene Okassa for useful comments and discussions. This work is an application of some results of my PhD prepared under joint superversion between University of Angers and University of Yaound\'e I. I would like to take this opportunity to thank my advisors, Vladimir Roubtsov and Bitjong Ndombol, for suggesting to me this interesting problem and for their availability during this project. I especially want to thank Larema for the logistics that he put at my disposal during this work. I also thank the French Ministry of Foreign Affairs, Franco-Cameroonian Cooperation, SARIMA and CIMPA for all their support and funding.

%%%%%%%%%%%%%%%%%%%%%%%%%%%%%%%%%%%%%%%%%%%%%%%%%%%%%%%%%%%%%%%%%%%%%%%
%%%%%%%%%%%%%%%%%%%%%%%%%%%%%%%%%%%%%%%%%%%%%%%%%%%%%%%%%%%%%%%%%%%%
%%%%%          The bibliography
%%%%%%%%%%%%%%%%%%%%%%%%%%%%%%%%%%%%%%%%%%%%%%%%%%%%%%%%%%%%%%%%%%%%%
%\begin{thebibliography}{11}
%%%%%%%%%%%%%%%%%%%%%%%%%%%%%%%%%%%%%%%%%%%%%%%%%%%%

%\end{thebibliography}
%%%%%%%%%%%%%%%%%%%%%%%%%%%%%%%%%%%%%%%%%%%%%%%%%%%%%%%%%%%%
%%%%%%%%%%%%%%%%%%%%%%%%%%%%%%%%%%%%%%%%%%%
%% The address
%\noindent
%      your name \\
%      your address 1\\
%      your address 2\\
%      etc.\\
%      your email
%%%%%%%%%%%%%%%%%%%%%%%%%%%%%%%%%%%%%%%%%%%%%%%%%%%%%%%%%%%%%%%%%

%%%%%%%%%%%%%%%%%%%%%%%%%%%%%%%%%%%%%%%%%%%
%%%%%%%%%%%%%  Leave the following command
%%%%%%%%%%%%%%%%%%%%%%%%%%%%%%%%%%%%%%%%%%%%%%%%%%%%
%\label{lastpage}
%\end{document}
\end{document}